\documentclass[12pt,a4paper]{amsart}
\pdfoutput=1

\usepackage{amsmath,amssymb,amsthm}  
\usepackage{mathtools}
\usepackage{tikz}
\usetikzlibrary{arrows}
\usepackage{thmtools, thm-restate}

\usepackage{xcolor} 	
\usepackage{hyperref}
\hypersetup{
	colorlinks,
    linkcolor={red!60!black},
    citecolor={green!60!black},
    urlcolor={blue!60!black},
}
\usepackage[abbrev, msc-links]{amsrefs} 
\usepackage{comment}

\usepackage[utf8]{inputenc}
\usepackage[T1]{fontenc}
\usepackage{lmodern}
\usepackage[babel]{microtype}
\usepackage[english]{babel}

\usepackage{geometry}
\usepackage{enumitem}

\theoremstyle{plain}
\newtheorem{thm}{Theorem}[section]

\newtheorem{cor}[thm]{Corollary}
\newtheorem{lem}[thm]{Lemma}

\newtheorem{quest}[thm]{Question}

\theoremstyle{definition}

\title{Greedoids from flames}

\author{Attila Jo\'{o}}
\thanks{The author would like to thank the generous support of the Alexander 
von Humboldt Foundation and NKFIH 
OTKA-129211}
\address{Attila Jo\'{o},
University of Hamburg, Department of Mathematics, Bundesstra{\ss}e 55 (Geomatikum), 20146 Hamburg, Germany}
\email{attila.joo@uni-hamburg.de}
\address{Attila Jo\'{o},
Alfr\'{e}d R\'{e}nyi Institute of Mathematics, Set theory and general topology research division, 13-15 Re\'{a}ltanoda St., 
Budapest, Hungary}
\email{jooattila@renyi.hu}

\begin{document}

\begin{abstract}
A digraph $ D $ with $ r\in V(D) $  is an $ r $-flame if for every $ {v\in V(D)-r} $, the  in-degree of $ v $ is  equal to 
the local edge-connectivity $ \lambda_D(r,v) $. We show that for every digraph $ D $ and $ r\in V(D) $, the  edge sets of the 
$ r $-flame 
subgraphs of $ D $ form a greedoid. Our method yields a new proof of Lovász' theorem stating: for every 
digraph $ D $ and $ r\in 
V(D) $, there is an  $ r $-flame subdigraph $ F $ of $ D $ such that  $ \lambda_F(r,v) =\lambda_D(r,v) $ for $ v\in V(D)-r $.  
We also give a strongly polynomial algorithm to find such an $ F $ working with a fractional generalization of Lovász' theorem.
\end{abstract}

\maketitle

\section{Introduction}

Subgraphs preserving some connectivity properties while having as few edges as possible have been a subject of 
interest since the beginning of graph theory. Suppose that $ D $ is a digraph with $ r\in V(D) $ and let us denote  the local 
edge-connectivity\footnote{The local edge-connectivity from $ r $ to  $ v $ is the maximal 
number of pairwise edge-disjoint $ r\rightarrow 
v $ 
paths.} from $ r $ to  some  $ v\in V(D)-r $ by $ \lambda_D(r,v) $.
We are looking 
for a spanning subgraph $ H $  of 
$ D $ with the smallest possible number of edges in which all the local edge-connectivities outwards from the root $ r $ are the 
same as in $ D $, i.e., $ 
\lambda_H(r,v)=\lambda_D(r,v) 
$ for all $ v\in V(D)-r $. In order to have $ \lambda_D(r,v) $ many pairwise edge-disjoint paths from $ r $ to $ v $ in $ H $, it is 
obviously necessary that the in-degree $ \varrho_H(v) $ of $ v $ in $ H $ is at least $ \lambda_D(r,v)  $. This leads to the 
estimation 
$ \left|E(H)\right|\geq \sum_{v\in V(D)-r}\lambda_D(r,v) $. It was shown by Lovász that, maybe surprisingly, this trivial 
lower bound is always sharp. 

\begin{thm}[Lov\'asz, Theorem 2 of \cite{lovasz}]\label{large flame}
For every digraph $ D $  and $ r\in V(D) $,  there is a spanning subdigraph $ H $ of $ D $ such that for every $ v\in 
V(D)-r $
\[  \lambda_D(r,v)=\lambda_{H}(r,v)= \varrho_H(v).\]
\end{thm}

Calvillo-Vives rediscovered Theorem \ref{large flame} independently in \cite{calvillo-vives} and  named the 
rooted digraphs $ F $ 
with $ \lambda_{F}(r,v)= \varrho_F(v) $ for all $ v\in V(F)-r $ `$ r $-flames' . 

We establish a direct connection between the extremal problem 
above and the theory of greedoids. The latter were introduced by Korte and Lovász as a generalization of matroids to capture 
greedy solvability in problems where 
the matroid concept turned out 
to be too restrictive.  The field is actively investigated since the '80s, for a survey we refer to \cite{GreedoidBook}.  

We show that the subflames of a 
rooted digraph always form a greedoid whose bases are exactly the subdigraphs described in Theorem \ref{large flame}.

\begin{thm}\label{flame greedoid}
Let $ D=(V,E) $ be a digraph and $ r\in V $. Then 
\[ \mathcal{F}_{D,r}:=\{ E(F)\, |\,   F\subseteq D \text{ is an }r\text{-flame} \} \]
is a greedoid on $ E $. Furthermore, for each $ \subseteq $-maximal element $ E(F^{*}) $ of 
$ \mathcal{F}_{D,r} $ we have $ \lambda_{F^{*}}(r,v)=\lambda_D(r,v) $ for all $  v\in V-r $.
\end{thm}

The proof of Theorem \ref{large flame} by Lov\'asz is algorithmic but only for simple digraphs polynomial. We prove a 
fractional generalization of Lovász' theorem considering digraphs with non-negative edge-capacities and replacing
`edge-connectivity' by `flow-connectivity'. Our proof provides a simple strongly polynomial algorithm to find an $ H$ with 
properties given in Theorem \ref{large flame}.

It is worth to mention that one can formulate a structural infinite generalization of Theorem \ref{large flame} in the same manner 
as Erdős conjectured such an extension of Menger's theorem (see \cite{aharoni2009menger}). As in the case of Menger's 
theorem, the problem is getting much harder in the infinite setting. The ``vertex-variant'' of this 
generalization  was proved for countably infinite digraphs in \cite{attila-flames} which was then further developed in 
\cite{erde2020enlarging}.

\section{Notation}
In this paper we deal only with finite combinatorial structures. An $ \mathcal{F}\subseteq 2^{E} $ is a greedoid on $ E $ if $ 
\varnothing\in \mathcal{F} $ and $ \mathcal{F} $ has the \emph{Augmentation property}, i.e.,
whenever $ F, F'\in 
\mathcal{F} $ with $ \left|F\right|<\left|F'\right| $, there is some $ e\in F'\setminus F $ such that $ F+e\in \mathcal{F} $.  
A digraph $ D $ is an ordered pair $ (V, E) $ where $ E $ is a set of directed edges with their endpoints in $ V $ where parallel 
edges are allowed but loops are not. Let us fix 
throughout this paper a vertex  set $ V $  and a ``root vertex'' $ r\in V $. For $ 
U\subseteq V $, $ \mathsf{in}_D(U) $ and $ \mathsf{out}_D(U) $ stand for the set of  ingoing and outgoing edges of $ U $ 
respectively, 
furthermore, let $ \varrho_D(U):=\left|\mathsf{in}_D(U)\right| $ and $ \delta_D(U):=\left|\mathsf{out}_D(U)\right| $. 
For simplicity we always assume that $ \mathsf{in}_D(r)=\varnothing $.  
We write shortly $ \lambda_D(v) $ for $ \lambda_D(r,v) $ where  $ v\in V-r $.  Recall, this is the
local edge-connectivity (i.e., the  maximal number of pairwise disjoint paths) from $ r $ to $ v $ . We 
define 
$ \mathcal{G}_{D}(v) $ to be the set of those  $ I \subseteq \mathsf{in}_D(v)  $ for which there exists a system $ \mathcal{P} $ 
of edge-disjoint $ r\rightarrow v $ paths where the set of the last edges of the paths in $ \mathcal{P} $ is $ I $. It is known that 
that set $ \mathcal{G}_{D}(v) $ is the family of independent sets of a matroid. Matroids representable this way were discovered 
by Perfect \cite{perfect1969independence}  and Pym \cite{pym1969proof} independently (using an equivalent definition based 
on  vertex-disjoint paths between vertex sets)  and are called 
gammoids. 
A digraph $ F $ is a flame if $ \mathcal{G}_{F}(v) $ is a free matroid\footnote{A free matroid is a matroid where all sets are 
independent.} for every 
$ v\in V-r $, equivalently $ 
\lambda_{F}(v)=\varrho_F(v) $ for every $ v\neq r $.
\section{The flame greedoid of a rooted digraph}

The core of the proof of Theorem \ref{flame greedoid} is the following lemma.
\begin{lem}\label{key lemma}
Let $ H $ and $ D $ be digraphs and assume that 
$\lambda_H(u)<\lambda_D(u)  $ for some $ u\in V-r $. Then there is an $e\in E(D)\setminus E(H) $ with head, say $ v $, such 
that 
$ e $ is a coloop\footnote{A coloop is an edge of a matroid which can be added to any independent set without ruin 
independence.} of  $ 
\mathcal{G}_{H+e}(v) $, i.e.,
\[ \mathcal{G}_{H+e}(v)\supseteq\{ I+e:I\in \mathcal{G}_{H}(v) \}.\]
\end{lem}
\begin{proof}
 Let $ \mathcal{U}:=\{ U\subseteq V-r: u\in U \text{ 
and } \varrho_H(U)=\lambda_{H}(u) \} $. By Menger's theorem 
$ \mathcal{U}\neq \varnothing $ and the submodularity of the map $ X \mapsto \varrho_H(X) $ ensures that $ 
\mathcal{U} $ is closed under union and intersection. Let $ U $ be the $ \subseteq $-largest element of $ \mathcal{U} $. Since
$\lambda_H(u)<\lambda_D(u)  $, there exists some edge $ e\in \mathsf{in}_{D}(U)\setminus \mathsf{in}_{H}(U)  $. Note 
that in $ H+e $ every  $ X\subseteq V-r $ with   $ X \supseteq U $ has at least $ 
\lambda_H(u)+1=\varrho_{H+e}(U) $ many ingoing 
edges because of the maximality of $ U $. 
By applying Menger's theorem in $ H+e $ with $ r $ and $ U $, we find a system $ \mathcal{P} $ of edge-disjoint  $ r\rightarrow 
U$ paths of size $ \lambda_H(u)+1 $ (see Figure \ref{Fig1}). The set 
of the last edges of the paths in $ \mathcal{P} $ is necessarily the whole $ \mathsf{in}_{H+e}(U) $. Let the 
head of $ e $ be $ v $ and let $ I\in \mathcal{G}_{H}(v) $ witnessed by the path-system $ \mathcal{Q} $. Clearly each $ Q\in 
\mathcal{Q} $ enters $ U $ at least once. For $ Q\in \mathcal{Q} $, we define $ f_{Q} $ as the last meeting of $ Q $ with 
$ \mathsf{in}_H(U) $. Finally, we build a path-system $ \mathcal{R} $ witnessing $ I+e\in \mathcal{G}_{H+e}(v) $ as follows. 
For $ Q\in 
\mathcal{Q} $, we consider the unique $ P_Q\in \mathcal{P} $ with last edge $ f_Q $ and concatenate it with the terminal 
segment of  $ Q $ from $ f_Q $ to obtain $ R_Q $. Moreover, let $ R_e $ be the unique path in $ \mathcal{P} $ with last edge $ 
e $. Then $ \mathcal{R}:=\{ R_Q: Q\in \mathcal{Q} \}\cup \{  R_e \} $ witnesses $ I+e\in   \mathcal{G}_{H+e}(v) $ as desired.

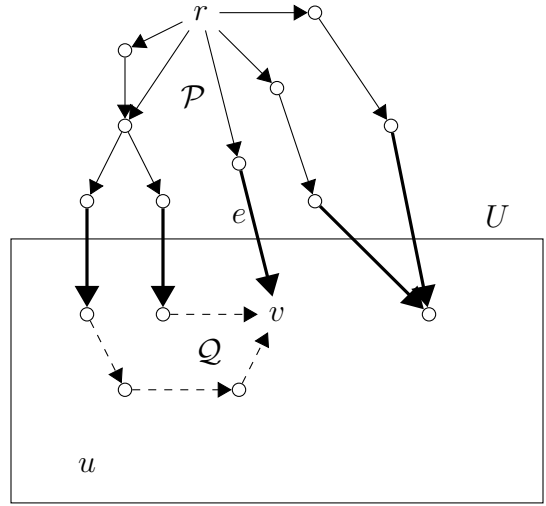
\begin{figure}[h]
\centering

\begin{tikzpicture}

\draw  (-2,0.5) rectangle (5,-3);
\node (v14) at (1.5,-0.5) {$v$};
\node at (-1,-2.5) {$u$};
\node (v1) at (0.5,3.5) {$r$};

\node[circle,inner sep=0pt,draw,minimum size=5] (v2) at (-0.5,3) {};
\node[circle,inner sep=0pt,draw,minimum size=5] (v3) at (-0.5,2) {};
\node[circle,inner sep=0pt,draw,minimum size=5] (v4) at (-1,1) {};
\node[circle,inner sep=0pt,draw,minimum size=5] (v5) at (-1,-0.5) {};
\draw [-triangle 60] (v1) edge (v2);
\draw [-triangle 60] (v2) edge (v3);
\draw [-triangle 60] (v3) edge (v4);
\draw [-triangle 60, very thick] (v4) edge (v5);

\node[circle,inner sep=0pt,draw,minimum size=5] (v6) at (0,1) {};
\node[circle,inner sep=0pt,draw,minimum size=5] (v7) at (0,-0.5) {};
\draw [-triangle 60] (v1) edge (v3);
\draw [-triangle 60] (v3) edge (v6);
\draw [-triangle 60, very thick] (v6) edge (v7);

\node[circle,inner sep=0pt,draw,minimum size=5]  (v8) at (1.5,2.5) {};
\node[circle,inner sep=0pt,draw,minimum size=5]  (v9) at (2,1) {};
\node[circle,inner sep=0pt,draw,minimum size=5] (v10) at (3.5,-0.5) {};
\draw [-triangle 60] (v1) edge (v8);
\draw [-triangle 60] (v8) edge (v9);
\draw [-triangle 60, very thick] (v9) edge (v10);

\node[circle,inner sep=0pt,draw,minimum size=5] (v11) at (2,3.5) {};
\node[circle,inner sep=0pt,draw,minimum size=5] (v12) at (3,2) {};
\draw [-triangle 60] (v1) edge (v11);
\draw [-triangle 60] (v11) edge (v12);
\draw [-triangle 60, very thick] (v12) edge (v10);

\node[circle,inner sep=0pt,draw,minimum size=5] (v13) at (1,1.5) {};
\node[circle,inner sep=0pt,draw,minimum size=5] (v15) at (-0.5,-1.5) {};
\node[circle,inner sep=0pt,draw,minimum size=5] (v16) at (1,-1.5) {};
\draw [-triangle 60] (v1) edge (v13);
\draw [-triangle 60, very thick] (v13) edge (v14);

\draw [-triangle 60, dashed] (v5) edge (v15);
\draw [-triangle 60, dashed] (v15) edge (v16);
\draw [-triangle 60, dashed] (v16) edge (v14);
\draw [-triangle 60, dashed] (v7) edge (v14);

\node at (1,0.8) {$e$};

\node at (4.4,0.8) {$U$};
\node at (0.4,2.4) {$\mathcal{P}$};
\node at (0.6,-1) {$\mathcal{Q}$};
\end{tikzpicture}
\caption{$ \mathsf{in}_{H+e}(U) $ consists of the thick edges,  the terminal segments of the paths in $ \mathcal{Q} $ are 
dashed. } \label{Fig1}

\end{figure}
\end{proof}
\begin{proof}[Proof of Theorem \ref{flame greedoid}]
Suppose that $ F_0, F_1\subseteq D $ are flames with $ \left|F_0\right|<  \left|F_1\right| $. Then there must 
be some $ u\in V-r $ for which $ \varrho_{F_0}(u)< \varrho_{F_1}(u)$. Since $ F_0 $ and $ F_1 $ 
are flames

 \[  \lambda_{F_0}(u)= \varrho_{F_0}(u)< \varrho_{F_1}(u)=\lambda_{F_1}(u).\]
 By applying Lemma \ref{key lemma} with $ F_0, F_1 $ and $ u $, we find an $ e\in E(F_1)\setminus (F_0) $ with head $ v $ 
 where $ 
 e $ is a coloop of $ 
 \mathcal{G}_{F_0+e}(v) $. On the one hand, $ \mathcal{G}_{F_0}(v) $ is a free matroid and  the previous 
 sentence ensures that $ \mathcal{G}_{F_0+e}(v) $ is free as well. On the other hand, for $ w\in V\setminus \{ r,v \} $ any 
 path-system 
 witnessing that $ \mathcal{G}_{F_0}(w) $ is a free matroid  shows the same for
 $ \mathcal{G}_{F_0+e}(w) $. By combining these we may conclude that $ F_0+e $ is a flame.
 
In order to prove the last sentence of Theorem \ref{flame greedoid}, let $ F^{*} $ be a maximal flame in $ D $ and suppose for a 
contradiction that 
$ \lambda_{F^{*}}(u)<\lambda_{D}(u)$ for some $ u\in V-r $. Applying Lemma \ref{key lemma} gives again some 
$ e\in E\setminus E(F^{*}) $ for which $ F^{*}+e $ is a flame contradicting the maximality of $ F^{*} $.
\end{proof}

\section{Fractional generalization and algorithmic aspects}
In this section we define a fractional version of Lovász's theorem and prove it by giving a strongly polynomial algorithm that finds 
a desired optimal substructure. We consider non-negative vectors indexed by 
the edge set $ E $ of a fixed digraph $ D=(V,E) $. This time we assume without loss of generality that $ D $ has no parallel edges 
because 
replacing a bunch of parallel 
edges by a single edge whose capacity is defined to be the sum of the capacities of those will be a meaningful reduction step in all 
the results we discuss.  For $ 
x,y\in 
\mathbb{R}_+^{E} $, we write $ x\leq y $ if $ x(e)\leq y(e) $ for every $ e\in E $ and for $ U\subseteq V $ let 
$\varrho_x(U):=\sum_{e\in \mathsf{in}_{D}(U)}x(e) $ and $ \delta_x(U):=\sum_{e\in \mathsf{out}_{D}(U)}x(e)  $. An $ x\in 
\mathbb{R}_+^{E}  $ is an $ r\rightarrow v $ flow if  $ \varrho_x(u)=\delta_x(u) $ holds for all $ u\in 
V\setminus \{ r,v \} $ and $ \varrho_x(r)=\delta_x(v)=0 $. We introduce some concepts and basic facts about flows, one can find 
more details and proofs for example in subsection 3.4 of \cite{frank2011connections}.  The \emph{amount} of the flow $ x $ is 
defined to be $ \delta_x(r) $ 
which is equal to $ \varrho_x(W)-\delta_x(W) $ for every choice of $ W\subseteq V-r $ containing $ v $.  Note that $ x $ can 
be written 
as the non-negative 
combination of directed cycles and $ r\rightarrow v $ paths (more precisely of their characteristic vectors). Such a decomposition 
can be found in a greedy way. The sum of the 
coefficients of the paths in any such a decomposition is again $ \delta_x(r) $.  For $ v\in V-r $ 
and $ c\in \mathbb{R}_+^{E}  $, the 
\emph{flow-connectivity} of $ c $
from $ r $ to $ v $ is
\[ \lambda_c(v):=\max \{ \delta_x(r)  : x\text{ is an }r\rightarrow v\text{ flow with }x\leq c. \} \]
The Max flow min cut theorem (see \cite{MFMC}) guarantees that  $ \lambda_c(v) $ is well-defined and equals to  
\[ \min \{ \varrho_c(W): W\subseteq V-r\text{ with }v\in W \}. \] For $v\in  V-r $ and $ c\in \mathbb{R}_+^{E}  $, we write
 $ \mathcal{G}_c(v) $ for the set of those vectors in $ \mathbb{R}_+^{\mathsf{in}_D(v)} $ that can be obtained as a restriction 
 of 
 an $  r\rightarrow v $ flow $ x\leq c $ to $ \mathsf{in}_D(v) $ that we denote by $ x \upharpoonright \mathsf{in}_D(v) $. It is 
 not too hard to prove that $ \mathcal{G}_c(v) $ is a polymatroid and it is natural to call it a \emph{polygammoid}.  An $ f\in 
 \mathbb{R}_+^{E}  $ is a \emph{fractional flame} 
if $ f\upharpoonright \mathsf{in}_D(v)\in  \mathcal{G}_f(v)  $ (equivalently $ \lambda_f(v)=\varrho_f(v) $) for all $ v\in V-r $.
For $ e\in E $, let $ \chi_e\in \mathbb{R}_+^{E}  $ be the vector where $ \chi_e(e')  $ is  $ 1 $ if $ e=e' $  and $ 0 $ 
otherwise.
We call a vector \emph{integral} if all of its coordinates are integers.

The fractional version of Lovász' theorem can be formulated in the following way.

\begin{thm}\label{large fractional flame Alg}
Let $ D=(V,E) $ be a digraph and $ r\in V $. Then for every  $ c\in \mathbb{R}_+^{E} $ there is an $ f\leq c $  such 
that for every $ v\in 
V-r $
\[  \lambda_c(v)=\lambda_f(v)=\varrho_f(v),\]
moreover, if   $ c $  is integral then $ f $ can be chosen to be integral. Such an $ f $ can be found in strongly polynomial time.
\end{thm}
\begin{proof}
In the contrast of Theorem \ref{large flame}, the following fractional analogue of Lemma \ref{key lemma} is not 
sufficient itself to provide the existence part of Theorem \ref{large fractional flame Alg} but will be an important tool later.

\begin{lem}\label{epsilon lemma}
Let $ x, y\in \mathbb{R}_+^{E} $ such that 
$\lambda_y(u)<\lambda_x(u)  $ for some $ u\in V-r $. Then there is an $e\in E $ with  head, say $ v $,  and 
an $ \varepsilon>0 $ such that $  
x(e)-y(e)\geq\varepsilon $  and
\[ \mathcal{G}_{y+\varepsilon \chi_e}(v)=\{s+\delta \chi_e :s\in \mathcal{G}_{y}(v)\wedge 0\leq\delta \leq \varepsilon 
\}.\]
\end{lem}
\begin{proof}
The proof goes similarly as for Lemma \ref{key lemma}. By applying the Max flow min cut theorem and the submodularity of the 
function $ X\mapsto \varrho_y(X) $, we take the maximal  $ U\subseteq V-r $ with $ u\in U $ and $ \varrho_y(U)=\lambda_y(u) 
$. We pick some $ e\in \mathsf{in}_D(U) $ with $ x(e)>y(e) $ and let 
\[ \varepsilon:= \min \{x(e)-y(e),\  \varrho_y(W)-\varrho_y(U): U\subsetneq W\subseteq V-r \}. \]
Let $ p $ be an $ r\rightarrow u $ flow of maximal amount 
with respect to the capacity  $ y+\varepsilon\chi_e $ in the auxiliary digraph we obtain by contracting $ U $ to $ u $  while 
deleting the arising loops.  By defining $ p $ on the edges with both ends in $ U $ to be $ 0 $, we ensure $ p\in 
\mathbb{R}_{+}^{E} $.
The Max flow min cut theorem and the choice of  $ \varepsilon $ guarantee  that  
 \[ p\upharpoonright \mathsf{in}_D(U)=(y+\varepsilon\chi_e) \upharpoonright \mathsf{in}_D(U). \] 
 We may assume that $ p $  is a non-negative combination of $ 
 r\rightarrow U $ paths.  Let $ s\in \mathcal{G}_y(u) $ witnessed by the $ r\rightarrow u $ flow $ q $ which is a non-negative 
 combination of $ 
 r\rightarrow u $ paths. Take the sum of the terminal segments of these weighted paths from the last common edge with $ 
 \mathsf{in}_D(U) $ together with the trivial path $ e $ with a given weight $ \delta $ with $ 0\leq \delta \leq \varepsilon$ to 
 obtain a vector $ q' $. 
 Starting with $ p $ one can construct a $ p'\leq p $ which is a non-negative combination of $ 
  r\rightarrow U $ paths and for which $p'\upharpoonright \mathsf{in}_D(U)=q' \upharpoonright \mathsf{in}_D(U) $. It is easy to 
  see that the coordinate-wise maximum of $ p' $ and $ q' $ witnessing $ s+\delta \chi_e\in \mathcal{G}_{y+\varepsilon 
  \chi_e}(v) $.
\end{proof}
Now we turn to the description of the algorithm. Let $ V=\{ v_0,\dots, v_n \} $ where $ v_0=r $. The algorithm starts with $ 
f_0:=c $. If $ f_k\in 
\mathbb{R}_+^{E}  $ is already 
constructed and $ k<n $, then we take an  $ r\rightarrow v_{k+1} $ flow $ z_{k+1}\leq f_k $ of amount $ 
\lambda_{f_k}(v_{k+1}) $, 
which we 
choose to 
be integral if 
 $ f_k $ is integral, and define  \[ f_{k+1}(e):= \begin{cases} z_{k+1}(e) &\mbox{if } e\in \mathsf{in}_D(v_{k+1}) \\
f_k(e) & \mbox{otherwise.} 
\end{cases}  \]

Since the flow problem can be solved in strongly polynomial time, the algorithm described above is strongly polynomial with a 
suitable flow-subroutine. We claim that $ f_n $ satisfies the demands of Theorem \ref{large fractional flame Alg}. Since we start 
with $ c $ and lower some values in each step, $ f_n\leq c $ holds. If $ c\in \mathbb{Z}_+^{E}  $, then a straightforward 
induction shows that  $ f_n\in \mathbb{Z}_+^{E}  $. 

\begin{lem}\label{side key lemma}
If $ z\leq x $ is an $ r\rightarrow v $ flow of amount $ \lambda_x(v) $  and $ {y(e):=
\begin{cases} z(e) &\mbox{if } 
e\in \mathsf{in}_D(v) \\
x(e) & \mbox{otherwise}\end{cases}}  $ then $ \lambda_y(u)=\lambda_x(u) $ for every $ u\in V-r $.
\end{lem}
\begin{proof}
Suppose for a contradiction that there exists a $ u\in V-r $ with $ \lambda_y(u)<\lambda_x(u) $. Note that $ u\neq v $ because $ 
\lambda_x(v)=\lambda_y(v) $ is witnessed by $ z $.  By 
Lemma \ref{epsilon lemma}, there is an $e\in E $  and 
an $ \varepsilon $ such that $  
x(e)-y(e)>\varepsilon>0 $ (which implies that the head of $ e $ must be $ v $) and
$  \mathcal{G}_{y+\varepsilon \chi_e}(v)=\{s+\delta \chi_e :s\in \mathcal{G}_{y}(v)\wedge 0\leq\delta \leq \varepsilon 
\} $.   Let $ s_0:= z \upharpoonright \mathsf{in}_D(v)  $.
 \[\lambda_x(v)\geq \lambda_{y+\varepsilon \chi_e}(v)\geq \left|\left|s_0 \right|\right|_1+\varepsilon=\lambda_x(v)+\varepsilon  \]
 which is a contradiction.
\end{proof}

By applying Lemma \ref{side key lemma} with $ x=f_k,\ y=f_{k+1} $ and $ z=z_{k+1} $ we  obtain the following.

\begin{cor}\label{keep trim}
$ \lambda_{f_k}(v)=\lambda_{f_{k+1}}(v) $ for every $ k<n $ and $ v\in V-r $. 
\end{cor}
It follows by induction on $ k $ that $ 
\lambda_{f_k}(v) =\lambda_c(v)$  for every $ v\in V-r $ and $ k\leq n $. In particular $ \lambda_{f_n}(v)=\lambda_c(v) $ for all 
$ v\in V-r $. Let $ 1\leq k \leq n $ be arbitrary. Then $ 
\varrho_{f_k}(v_k)=\lambda_{f_k}(v_k) $ follows directly from the algorithm (the common value is $ \varrho_{z_k}(v_k) $). 
On the one hand, the left side is equal to $ \varrho_{f_n}(v_k) $ since in moving from $ f_k $ to $ f_n $ the algorithm no longer 
changes the values on the elements of $ \mathsf{in}_D(v_k) $. On the other hand,  we have seen that $ 
\lambda_{f_k}(v_k)=\lambda_{f_n}(v_k)=\lambda_c(v_k) $. By combines 
these we have $ \varrho_{f_n}(v)=\lambda_{f_n}(v) $  which completes the proof of Theorem \ref{large fractional flame Alg}.
\end{proof}

Finally, let us point out a special case of Lemma \ref{side key lemma}.
\begin{cor}
Let $ D $ be a directed graph and let $ \mathcal{P} $ be a maximal sized family of pairwise edge-disjoint $ r\rightarrow v $ paths 
in $ D $. Then the deletion of those ingoing edges of $ v $ that are unused by the path-family $ \mathcal{P} $ does not reduce 
any local 
edge-connectivities of the form $ \lambda_D(r,u) $ with $ u\in V(D)-r $.
\end{cor}
\section{Outlook}

By Theorem \ref{large fractional flame Alg}, finding a spanning subdigraph of a given digraph $ D $ that preserves all the local 
edge-connectivities from a prescribed root vertex $ r $ and has the fewest possible edges with respect to this property can be done 
in polynomial time. It is natural to ask the complexity of the weighted version:

\begin{quest}
What is the complexity of the following combinatorial optimization problem?\\
Input: digraph $ D $, $ r\in V(D) $ and cost function $ c: E(D)\rightarrow \mathbb{R}_+ $\\
Output: spanning subdigraph $ F $ of $ D $ with $ \lambda_F(r,v)=\lambda_D(r,v) $ for every $ v\in V(D)-r $ for which $ 
\sum_{e\in E(F)}c(e) $ is minimal with respect to this property.
\end{quest}
The special case where $ \lambda_D(r,v)  $ is the same for every $ v\in V(D)-r $ can be solved in polynomial time by using 
weighted matroid intersection (see \cite{MR0270945}).\\

There are more general flow models involving polymatroidal bounding functions (see for example \cite{polyflow} and 
\cite{quasipolyflow}). The 
Max flow min 
cut theorem is preserved under these models.

\begin{quest}
Is it possible to generalize Theorem \ref{large fractional flame Alg} by using the polymatroidal flow model introduced by  
Hassin in \cite{hassin1982minimum} (and rediscovered later  by Lawler and Martel 
in \cite{polyflow} independently)?
\end{quest}
\

The relation between matroids and polymatroids motivates the following concept of polygreedoids:
a \emph{polygreedoid} is a compact $ \mathcal{P}\subseteq \mathbb{R}_+^{E} $ such that 
\begin{enumerate}
\item[PG1] $ \underline{0}\in \mathcal{P} $,
\item[PG2]  whenever $ x,y\in 
\mathcal{P} $ with $\left|\left|x\right|\right|_1<\left|\left|y\right|\right|_1$, there is some $ e\in E $ with $ y(e)>x(e) $ such that $ 
x+\varepsilon \chi_e\in \mathcal{P} $ for all small enough $ \varepsilon>0 $.
\end{enumerate}

It follows directly from Lemma \ref{epsilon lemma} that fractional flames under a given bounding vector form a polygreedoid.
Greedoids have a property called \emph{accesibility} which can be considered as a weakening of the downward closedness of 
matroids. It tells that every $ F\in \mathcal{F} $ can be enumerated in such a way that each initial segment belongs to $ 
\mathcal{F} $, i.e.,  $ F=\{ e_1,\dots, e_n \} $ such that $ \{ e_1,\dots, e_k \}\in \mathcal{F} $ for every $ k\leq n $. 
Accessibility tends to be a part of the axiomatization of greedoids via the restriction the Augmentation axiom for pairs with 
$ \left|F'\right|=\left|F\right|+1 $. It is not too hard to prove that polygreedoids satisfy the following analogous property:  for every 
$ x\in \mathcal{P} $ there is a continues strictly increasing\footnote{Strictly increasing is meant with respect to the 
coordinate-wise partial ordering of $ 
\mathbb{R}_+^{E} $.} function $ g:[0,1]\rightarrow 
\mathcal{P} $ with $ g(0)= \underline{0}$ and $ g(1)=x $. Finally, let us end the paper with the following general question.

\begin{quest}
How much of the theory of greedoids is preserved for polygreedoids?
\end{quest}


\begin{bibdiv}
\begin{biblist}




\bib{lovasz}{article}{
   author={Lov\'{a}sz, L.},
   title={Connectivity in digraphs},
   journal={J. Combinatorial Theory Ser. B},
   volume={15},
   date={1973},
   pages={174--177},
   issn={0095-8956},
   review={\MR{325439}},
   doi={10.1016/0095-8956(73)90018-x},
}


\bib{calvillo-vives}{thesis}{
	author={Calvillo-Vives, Gilberto}, 
	title={Optimum branching systems}, 
	date={1978},
	type={Ph.D. Thesis},
	organization={University of Waterloo},
}

\bib{Greedoid Book}{book}{ author={Korte, Bernhard}, author={Lov\'{a}sz, L\'{a}szl\'{o}}, author={Schrader, Rainer}, 
title={Greedoids}, series={Algorithms and Combinatorics}, volume={4}, publisher={Springer-Verlag, Berlin}, date={1991}, 
pages={viii+211}, isbn={3-540-18190-3}, review={\MR{1183735}}, doi={10.1007/978-3-642-58191-5},}

\bib{aharoni2009menger}{article}{
   author={Aharoni, Ron},
   author={Berger, Eli},
   title={Menger's theorem for infinite graphs},
   journal={Invent. Math.},
   volume={176},
   date={2009},
   number={1},
   pages={1--62},
   issn={0020-9910},
   review={\MR{2485879}},
   doi={10.1007/s00222-008-0157-3},
}

\bib{attila-flames}{article}{
   author={Jo\'{o}, Attila},
   title={Vertex-flames in countable rooted digraphs preserving an Erd\H{o}s-Menger separation for each vertex},
   journal={Combinatorica},
   volume={39},
   date={2019},
   pages={1317--1333},
   doi={10.1007/s00493-019-3880-z},
}

\bib{erde2020enlarging}{article}{
    title={Enlarging vertex-flames in countable digraphs},
    author={Joshua Erde and J. Pascal Gollin and Attila Jo\'{o}},
    year={2020},
    journal={arXiv preprint arXiv:2003.06178},
    note={\url{https://arxiv.org/abs/2003.06178v1}}
}

\bib{perfect1969independence}{article}{
    title={Independence spaces and combinatorial problems},
     author={Perfect, Hazel},
     journal={Proceedings of the London Mathematical Society},
     volume={3},
     number={1},
     pages={17--30},
     year={1969},
     publisher={Wiley Online Library}
}

\bib{pym1969proof}{article}{
    title={A proof of the linkage theorem},
      author={Pym, JS},
      journal={Journal of Mathematical Analysis and Applications},
      volume={27},
      number={3},
      pages={636--638},
      year={1969},
      publisher={Elsevier}
    }

\bib{frank2011connections}{book}{
  title={Connections in combinatorial optimization},
  author={Frank, Andr{\'a}s},
  volume={38},
  year={2011},
  publisher={OUP Oxford}
}


\bib{MFMC}{book}{ author={Ford, L. R., Jr.}, author={Fulkerson, D. R.}, title={Flows in networks}, 
publisher={Princeton University Press, Princeton, N.J.}, date={1962}, pages={xii+194}, review={\MR{0159700}},}

\bib{MR0270945}{article}{
   author={Edmonds, Jack},
   title={Submodular functions, matroids, and certain polyhedra},
   conference={
      title={Combinatorial Structures and their Applications},
      address={Proc. Calgary Internat. Conf., Calgary, Alta.},
      date={1969},
   },
   book={
      publisher={Gordon and Breach, New York},
   },
   date={1970},
   pages={69--87},
   review={\MR{0270945}},
}


\bib{hassin1982minimum}{article}{
  title={Minimum cost flow with set-constraints},
  author={Hassin, Refael},
  journal={Networks},
  volume={12},
  number={1},
  pages={1--21},
  year={1982},
  publisher={Wiley Online Library}
}

\bib{polyflow}{article}{
   author={Lawler, E. L.},
   author={Martel, C. U.},
   title={Polymatroidal flows with lower bounds},
   note={Applications of combinatorial methods in mathematical programming
   (Gainesville, Fla., 1985)},
   journal={Discrete Appl. Math.},
   volume={15},
   date={1986},
   number={2-3},
   pages={291--313},
   issn={0166-218X},
   review={\MR{865009}},
   doi={10.1016/0166-218X(86)90050-8},
}
		
\bib{quasipolyflow}{article}{
   author={Kochol, M.},
   title={Quasi polymatroidal flow networks},
   journal={Acta Math. Univ. Comenian. (N.S.)},
   volume={64},
   date={1995},
   number={1},
   pages={83--97},
   issn={0862-9544},
   review={\MR{1360989}},
}
	


\end{biblist}
\end{bibdiv}
\end{document}